\documentclass[12pt]{amsart}
\usepackage[utf8]{inputenc}
\usepackage{amssymb}
\usepackage{hyperref}
\usepackage[final]{showkeys} 

\input xy
\xyoption{all}

\theoremstyle{definition}
\newtheorem{mydef}{Definition}[section]

\newtheorem{thm}[mydef]{Theorem}

\newcommand{\footnotei}[1]{}

\newcommand{\cf}[1]{\text{cf} (#1)}
\newcommand{\seq}[1]{\langle #1 \rangle}

\def\lea{\le}

\def\gea{\ge}

\newcommand{\K}{K}

\newbox\noforkbox \newdimen\forklinewidth
\forklinewidth=0.3pt \setbox0\hbox{$\textstyle\smile$}
\setbox1\hbox to \wd0{\hfil\vrule width \forklinewidth depth-2pt
 height 10pt \hfil}
\wd1=0 cm \setbox\noforkbox\hbox{\lower 2pt\box1\lower
2pt\box0\relax}
\def\unionstick{\mathop{\copy\noforkbox}\limits}

\def\1nf{\unionstick^{(1)}}

\def\2nf{\unionstick^{(2)}}
\def\3nf{\unionstick^{(3)}}

\newcommand{\gtp}{\text{gtp}}

\newcommand{\gS}{\text{gS}}

\newcommand{\hanf}[1]{h (#1)}

\newcommand{\cl}{\text{cl}}

\newcommand{\F}{\mathcal{F}}
\newcommand{\LS}{\text{LS}}

\title[Guide to categoricity in universal classes]{The lazy model theoretician's guide to {S}helah's eventual categoricity conjecture in universal classes}
\date{\today}

\author{Sebastien Vasey}
\email{sebv@cmu.edu}
\urladdr{http://math.cmu.edu/\textasciitilde svasey/}
\address{Department of Mathematical Sciences, Carnegie Mellon University, Pittsburgh, Pennsylvania, USA}
\thanks{The author is supported by the Swiss National Science Foundation.}

\begin{document}

\begin{abstract}
  We give a short overview of the proof of Shelah's eventual categoricity conjecture in universal classes with amalgamation \cite{ap-universal-v9}.
\end{abstract}

\maketitle

\section{Introduction}

We sketch a proof of:

\begin{thm}\label{main-thm}
  Let $\K$ be a universal class with amalgamation. If $\K$ is categorical in\footnote{Here and below, we write $\hanf{\theta} := \beth_{(2^{\theta})^+}$. We see universal classes as AECs so that for $K$ a universal class, $\LS (\K) = |L (\K)| + \aleph_0$. For $\K$ a fixed AEC, we write $H_1 := \hanf{\LS (\K)}$ and $H_2 := \hanf{H_1}$.} \emph{some} $\lambda > H_2$, then $\K$ is categorical in \emph{all} $\lambda' \ge H_2$.
\end{thm}

The reader should see the introduction of \cite{ap-universal-v9} for motivation and history. Note that (as stated there) the amalgamation hypothesis can be removed assuming categoricity in cardinals of arbitrarily high cofinality. However this relies on hard arguments of Shelah \cite[Chapter IV]{shelahaecbook}, so we do not discuss it. There are plans for a sequel where the amalgamation hypothesis will be removed under categoricity in a single cardinal of arbitrary cofinality (earlier versions actually claimed it but the argument contained a mistake).

Note that this is not a self-contained argument, we simply attempt to outline the proof and quote extensively from elsewhere. For another exposition, see the upcoming \cite{bv-survey}. 

We attempt to use as few prerequisites as possible and make what we use explicit. We do not discuss generalizations to tame AECs with primes \cite{categ-primes-v3}, although we end up using part of the proof there. 

We assume familiarity with a basic text on AECs such as \cite{baldwinbook09} or the upcoming \cite{grossbergbook}. We also assume the reader is familiar with the definition of a good $\F$-frame (see \cite[Chapter II]{shelahaecbook} for the original definition of a good $\lambda$-frame and \cite[Definition 2.21]{ss-tame-toappear-v3} for good $\F$-frames), and the definition of superstability (implicit in \cite{shvi635}, but we use the definition in \cite[Definition 10.1]{indep-aec-v5}). All the good frames we will use are \emph{type-full}, i.e.\ their basic types are the nonalgebraic types, and we will omit the ``type-full''.

This note was written while working on a Ph.D.\ thesis under the direction of Rami Grossberg at Carnegie Mellon University and I would like to thank Professor Grossberg for his guidance and assistance in my research in general and in this work specifically. I thank John Baldwin for early feedback on this note.

\section{The proof}

The argument depends on \cite{sh394}, on the construction of a good frame and related results in \cite{ss-tame-toappear-v3}, on Boney's theorem on extending good frames using tameness \cite{ext-frame-jml} (the subsequent paper \cite{tame-frames-revisited-v4} is not needed here), and on the Grossberg-VanDieren categoricity transfer \cite{tamenesstwo}. The argument also depends on some results about unidimensionality in III.2 of \cite{shelahaecbook} (these results have short full proofs, and have appeared in other forms elsewhere, most notably in \cite{tamenesstwo,tamenessthree}).

 There is a dependency on the Shelah-Villaveces theorem (\cite[Theorem 2.2.1]{shvi635}), which can be removed in case one is willing to assume that $\cf{\lambda} > \LS (\K)$. This is reasonable if one is willing to assume that $K$ is categorical in unboundedly many cardinals: then by amalgamation, the categoricity spectrum will contain a club, hence cardinals of arbitrarily high cofinality.

\begin{proof}[Proof of Theorem \ref{main-thm}] 
  We proceed in several steps.
  \begin{enumerate}
    \item Without loss of generality, $\K$ has joint embedding and no maximal models.

      [Why? Let us define a relation $\sim$ on $\K$ by $M \sim N$ if and only if $M$ and $N$ embed into a common extension. Using amalgamation, one can see that $\sim$ is an equivalence relation. Now the equivalence classes $\seq{\K_i : i \in I}$ of $\sim$ form disjoint AECs with amalgamation and joint embedding, and by the categoricity assumption (recalling that the Hanf number for existence is bounded by $H_1$) there is a unique $i \in I$ such that $\K_i$ has arbitrarily large models. Moreover $(\K_i)_{\ge H_1} = \K_{\ge H_1}$ so it is enough to work inside $\K_i$.]
    \item $\K$ is $\LS (\K)$-superstable. 

      [Why? By \cite[Theorem 2.2.1]{shvi635}, or really the variation using amalgamation stated explicitly in \cite[Theorem 6.3]{gv-superstability-v2}. Alternatively, if one is willing to assume that $\cf{\lambda} > \LS (\K)$, one can directly apply \cite[Lemma 6.3]{sh394}.]
    \item $\K$ is $(<\aleph_0)$-tame.

      [Why? See \cite[Section 3]{ap-universal-v9}\footnote{The main idea there is due to Will Boney, see \cite{tameness-groups}.} (this does not use the categoricity hypothesis).]
    \item $\K$ is stable in $\lambda$.

      [Why? By \cite[Theorem 5.6]{ss-tame-toappear-v3}, $\LS (\K)$-superstability and $\LS (\K)$-tameness imply stability everywhere above $\LS (\K)$.]
    \item\label{sat-step} The model of size $\lambda$ is saturated.

      [Why? Use stability to build a $\mu^+$-saturated model of size $\lambda$ for each $\mu < \lambda$. Now apply categoricity.]
    \item $\K$ is categorical in $H_2$. 

      [Why? By the proof of \cite[II.1.6]{sh394}, or see \cite[14.8]{baldwinbook09}.]
      
    \item $\K$ has a good $H_2$-frame.

      [Why? By \cite[Theorem 7.3]{ss-tame-toappear-v3} which tells us how to construct a good frame at a categoricity cardinal assuming tameness and superstability below it.]
    \item For $M \in \K_{H_2}$, $p \in \gS (M)$, let $\K_{\neg^\ast p}$ be defined as in \cite[Definition 5.7]{ap-universal-v9}: roughly, it is the class of $N$ so that $p$ has a unique extension to $\gS (N)$ (so in particular $p$ is omitted in $N$), but we add constant symbols for $M$ to the language to make it closed under isomorphisms. Then $\K_{\neg^\ast p}$ is a universal class.

      [Why? That it is closed under substructure is clear. That it is closed under unions of chains is because universal classes are $(<\aleph_0)$-tame, so if a type has two distinct extensions over the union of a chain, it must have two distinct extension over an element of the chain. Here is an alternate, more general, argument: $\K_{H_2}$ is $\aleph_0$-local (by the existence of the good frame), so using tameness it is not hard to see that $\K_{\ge H_2}$ is $\aleph_0$-local. Now proceed as before.]
    \item If $K$ is not categorical in $H_2^+$, then there exists $M \in \K_{H_2}$ and $p \in \gS (M)$ so that $\K_{\neg^\ast p}$ has a good $H_2$-frame.

      [Why? See \cite[Theorem 2.15]{categ-primes-v3}\footnote{The original argument in \cite{ap-universal-v9} is harder, as it requires building a global independence relation.}: it shows that if $K_{H_2}$ is weakly unidimensional (a property that Shelah introduces in III.2 of \cite{shelahaecbook} and shows is equivalent to categoricity in $H_2^+$), then the good $H_2$-frame that $K$ has, restricted to $\K_{\neg^\ast p}$ (for a suitable $p$) is a good $H_2$-frame. The definition of weak unidimensionality is essentially the negation of the fact that there exists two types $p \perp q$ (for a notion of orthogonality defined using prime models).]
    \item If $K$ is not categorical in $H_2^+$, $K_{\neg^\ast p}$ above has arbitrarily large models.

      [Why? By Theorem \ref{step-3} below (recalling that $\K_{\neg^\ast p}$ is a universal class), $\K_{\neg^\ast p}$ has a good $(\ge H_2)$-frame. Part of the definition of such a frame requires existence of a model in every cardinal $\mu \ge H_2$.
    \item If $K$ is not categorical in $H_2^+$, the model of size $\lambda$ is not saturated. This contradicts (\ref{sat-step}) above, therefore $K$ is categorical in $H_2^+$.

      [Why? Take $M \in \K_{\neg^\ast p}$ of size $\lambda$ (exists by the previous step). Then $M$ omits $p$ and the domain of $p$ has size $H_2 < \lambda$.]
    \item $K$ is categorical in all $\lambda' \ge H_2$.

      [Why? We know that $K$ is categorical in $H_2$ and $H_2^+$, so apply the upward transfer of Grossberg and VanDieren \cite[Theorem 0.1]{tamenesstwo}.
  \end{enumerate}
\end{proof}

To complete the proof, we need the following:

\begin{thm}\label{step-3}
  Let $K$ be a universal class. Let $\lambda \ge \LS (\K)$. If $\K$ has a good $\lambda$-frame, then $\K$ has a good $(\ge \lambda)$-frame.
\end{thm}
\begin{proof} \
  \begin{enumerate}
    \item $K$ is $\lambda$-tame for types of length two.

      [Why? See \cite[Section 3]{ap-universal-v9}.]
    \item $K$ has weak amalgamation: if\footnote{Since we do not assume amalgamation, Galois types are defined using the transitive closure of atomic equivalence, see e.g.\ \cite[Definition II.1.9]{shelahaecbook}.} $\gtp (a_1 / M; N_1) = \gtp (a_2 / M; N_2)$, there exists $N_1' \lea N_1$ containing $a_1$ and $M$ and $N \gea N_1'$, $f: N_2 \xrightarrow[M]{} N$ so that $f (a_2) = a_1$.

      [Why? By the isomorphism characterization of Galois types in AECs which admit intersections, see \cite[Lemma 2.6]{non-locality} or \cite[Proposition 2.17]{ap-universal-v9}. More explicitly, set $N_1' := \cl^{N_1} (a_1 M)$, where $\cl^{N_1}$ denotes closure under the functions of $N_1$. Then chase the definition of equality of Galois types.]
      
    \item $K$ has amalgamation. 

      [Why? By \cite[Theorem 4.15]{ap-universal-v9}.]

    \item $K$ has a good $(\ge \lambda)$-frame.

      [Why? By Boney's upward frame transfer \cite{ext-frame-jml} which tells us that amalgamation, $\lambda$-tameness for types of length two, and a good $\lambda$-frame imply that the frame can be extended to a good $(\ge \lambda)$-frame.]
  \end{enumerate}
\end{proof}

\bibliographystyle{amsalpha}
\bibliography{uc-categ-overview}

\end{document}